\documentclass[leqno,12pt]{article} 
\usepackage{amsmath, amssymb}
\usepackage{amsthm} 
\usepackage{braket}
\usepackage{mathtools}

\usepackage{latexsym}
\def\qed{\hfill $\Box$}
\usepackage{cases}

\makeatletter
    
    \@addtoreset{equation}{section}
\makeatother

\theoremstyle{plain} 
\newtheorem{theorem}{\noindent\bf Theorem}[section] 
\newtheorem{lemma}[theorem]{\noindent\bf Lemma}
\newtheorem{corollary}[theorem]{\noindent\bf Corollary}

\theoremstyle{definition} 

\newtheorem{remark}[theorem]{\indent\bf Remark}

\newcommand{\End}[0]{\operatorname{End}}

\newcommand{\Ker}[0]{\operatorname{Ker}}

\newcommand{\deldel}{\sqrt{-1}\partial \overline{\partial}}
\newcommand{\dbar}{\overline{\partial}}
\newcommand{\e}{\varepsilon}

\newcommand{\lla}[0]{{\langle\!\hspace{0.02cm} \!\langle}}
\newcommand{\rra}[0]{{\rangle\!\hspace{0.02cm}\!\rangle}}

\makeatletter

\def\address#1#2{\begingroup
\noindent\parbox[t]{7.8cm}{
\small{\scshape\ignorespaces#1}\par\vskip1ex
\noindent\small{\itshape E-mail}
\/: #2\par\vskip4ex}\hfill
\endgroup}
\makeatother

\title{Hessian of the Ricci Calabi functional}

\author{
\textsc{Satoshi Nakamura$^{*}$}
} 

\date{}

\begin{document}

\maketitle


\footnote{ 
$^{*}$Partly supported by Grant-in-Aid for JSPS Fellowships for Young Scientists, Number 17J02783.
}

\begin{abstract}
The Ricci Calabi functional is a functional on the space of K\"ahler metrics of Fano manifolds. Its critical points are called generalized K\"ahler Einstein metrics. In this article, we show that the Hessian of the Ricci Calabi functional is non-negative at  generalized K\"ahler Einstein metrics.
As its application, we give another proof of a Matsushima's type decomposition theorem for holomorphic vector fields, which was originally proved by Mabuchi.
We also discuss a relation to the inverse Monge-Amp\`ere flow developed recently by Collins-Hisamoto-Takahashi.
\end{abstract}


\section{Introduction}
In his paper \cite{Ma01}, Mabuchi extended the notion of K\"{a}hler Einstein metrics for Fano manifolds with non-vanishing Futaki invariant.
In this paper, we call them generalized K\"{a}hler Einstein metrics.
Let $X$ be an $n$-dimensional Fano manifold and $\omega \in 2\pi c_1(X)$ be a reference
K\"{a}hler metric.
We denote its volume $\int_X\omega^n$  by $V$.
Let $$\mathcal{M}(\omega) :=\Set{\phi\in C_{\mathbb{R}}^{\infty}(X) | \omega_{\phi}:= \omega +\deldel\phi >0}$$
be the space of K\"ahler metrics in $[\omega]=2\pi c_1(X).$
We usually identify the K\"ahler metric $\omega_{\phi}$ with its potential $\phi\in\mathcal{M}(\omega)$.
Note that the tangent space $T_{\phi}\mathcal{M}(\omega)$ is nothing but $C^{\infty}_{\mathbb{R}}(X).$
We denote the Ricci form for $\phi$ by $\mathrm{Ric}(\omega_{\phi})= -\deldel\log\det\omega_{\phi}^n$.
The Ricci potential $f_{\phi}$ for $\phi$ is a function satisfying
\begin{equation}\label{Riccipotential}
\mathrm{Ric}(\omega_{\phi})-\omega_{\phi}=\deldel f_{\phi} \quad\text{and}\quad \int_X e^{f_{\phi}}\omega_{\phi}^n =V.
\end{equation}
Then $\omega_{\phi} = \sqrt{-1}g_{i \bar{j}}dz^i\wedge d\bar{z}^j$ is called {\it generalized K\"{a}hler Einstein} 
if the complex gradient vector field of $1-e^{f_{\phi}}$ is holomorphic, that is,
$$\dbar \Bigl(g^{i\bar{j}} \frac{\partial (1-e^{f_{\phi}})}{\partial\bar{z}^j} \frac{\partial}{\partial z^i} \Bigr)=0.$$
If $X$ has no nontrivial holomorphic vector field, generalized K\"ahler Einstein metrics are nothing but K\"ahler Einstein metrics.
In general, if the Futaki invariant of $X$ vanishes, these are K\"ahler Einstein metrics.

Yao \cite{Ya} gave a characterization of generalized K\"ahler Einstein metrics in terms of the {\it Ricci Calabi functional} (Yao calls it the Ding energy).
The Ricci Calabi functional $\mathcal{E}_{\mathrm{RC}}$ is a functional on $\mathcal{M}(\omega)$ defined by
$$\mathcal{E}_{\mathrm{RC}}(\phi) = \int_X (1-e^{f_{\phi}})^2 \omega_{\phi}^n.$$
Yao observed that $\phi\in\mathcal{M}(\omega)$ is a critical point of $\mathcal{E}_{\mathrm{RC}}$ 
if and only if it is a generalized K\"ahler Einstein metric (Yao calls it the Mabuchi metric). See section \ref{variation}.
Following Donaldson's new GIT (Geometric Invariant Theory) picture \cite{Do17} for Fano manifolds, 
this functional $\mathcal{E}_{\mathrm{RC}}$ can be seen as the norm squared of a moment map on the corresponding moduli space.
This shows that generalized K\"ahler Einstein metrics can be seen as one of the canonical K\"ahler metrics for Fano manifolds.
Note that they are not in general neither extremal K\"ahler metrics nor K\"ahler Ricci solitons.

The main result of this paper is a development of the observation by Yao \cite{Ya}.
In fact, we show that generalized K\"ahler Einstein metrics are local minimums of the Ricci Calabi functional. 
To state results more precisely, we fix more notations.
We denote the (negative) Laplacian of $\omega_{\phi}\in\mathcal{M}(\omega)$ by $\Delta_{\phi}$.
Let us define $$Lu=\Bigl( -\Delta_{\phi}u -\langle \dbar u, \dbar f_{\phi}\rangle -u + \frac{1}{V}\int_X u e^{f_{\phi}}\omega_{\phi}^n \Bigr)e^{f_{\phi}}$$
as an operator on $C_{\mathbb{C}}^{\infty}(X)$. 
We define its complex conjugate operator by $\overline{L}u:=\overline{L\overline{u}}.$
More precisely we have
$$\overline{L}u=\Bigl( -\Delta_{\phi}u -\langle \partial u, \partial f_{\phi} \rangle -u + \frac{1}{V}\int_X u e^{f_{\phi}}\omega_{\phi}^n \Bigr)e^{f_{\phi}}.$$
We define a natural inner product on $C_{\mathbb{C}}^{\infty}(X)$ by $$\lla u,v \rra = \int_X u\overline{v} \omega_{\phi}^n.$$
Note that operators $L$ and $\overline{L}$ and the inner product $\lla \cdot, \cdot \rra$ depend on a K\"ahler metric $\phi\in\mathcal{M}(\omega)$.

The followings are main results of this paper.
\begin{theorem}\label{Hessian1}
The Hessian of the Ricci  Calabi functional $\mathcal{E}_{\mathrm{RC}}$ at every generalized K\"ahler Einstein metric $\phi$ along directions $\delta\phi_1, \delta\phi_2\in T_{\phi}\mathcal{M}(\omega)$ is given by
$$\mathrm{Hess}(\mathcal{E}_{\mathrm{RC}})(\delta\phi_1, \delta\phi_2)
=2\lla L\overline{L}\delta\phi_1,\delta\phi_2\rra=2\lla \overline{L}L\delta\phi_1,\delta\phi_2\rra.$$
\end{theorem}
As a corollary, we can see that the Hessian $\mathrm{Hess}(\mathcal{E}_{\mathrm{RC}})$ is non-negative at every generalized K\"ahler Einstein metric.
\begin{corollary}\label{Hessian2}
At every generalized K\"ahler Einstein metric, operators $L$ and $\overline{L}$ are commutative.
As the result, their composition $L\overline{L}$ is a self-adjoint non-negative operator on $T_{\phi}\mathcal{M}(\omega)$ with respect to the inner product $\lla \cdot, \cdot \rra.$
\end{corollary}
Similar results for the Hessian formula and the non-negativity of the Hessian were observed for various functionals and its critical K\"ahler metrics.
See for instance, \cite{Ca85, Wa06, Gabook} (for the Calabi functional and the extremal K\"ahler metric),  \cite{Fu08} (for a Calabi's type functional and the perturbed extremal K\"ahler metric), \cite{F16} (for the He functional and the K\"ahler Ricci soliton) and \cite{FO17} (for a Calabi's type functional and the $f$-extremal K\"ahler metric).

Let $\mathfrak{h}(X)$ be the the Lie algebra of holomorphic vector fields on $X$.
For any $u\in C^{\infty}_{\mathbb{C}}(X)$, we define the gradient vector field $\mathrm{grad}_{\phi}u$ for a K\"ahler metric $\phi\in\mathcal{M}(\omega)$ by
$$i_{\mathrm{grad}_{\phi}u}\omega_{\phi}=\sqrt{-1}\dbar u.$$
As an application of Corollary \ref{Hessian2}, we give the following Matsushima's type decomposition theorem for $\mathfrak{h}(X)$.
\begin{theorem}\label{MaLi}
Let $X$ be a Fano manifolds admitting a generalized K\"ahler Einstein metric $\phi\in\mathcal{M}(\omega)$.
Then the Lie algebra $\mathfrak{h}(X)$ is, as a vector space, the direct sum
$$\mathfrak{h}(X)=\sum_{\lambda\geq 0}\mathfrak{h}_{\lambda}(X),$$
where $\mathfrak{h}_{\lambda}(X)$ is the $\lambda$-eigenspace of the adjoint action of $-\mathrm{grad}_{\phi}e^{f_{\phi}}$.
Furthermore, $\mathfrak{h}_{0}(X)$ is the complexification of the Lie algebra of Killing vector fields on $(X,\omega_{\phi})$.
In particular, $\mathfrak{h}_{0}(X)$ is reductive.
\end{theorem}
This theorem was originally proved by Mabuchi in \cite[Theorem 4.1]{Ma01}.
His proof heavily depends on Futaki-Mabuchi's theory \cite{FM} of the extremal K\"ahler vector field.
Our proof is done by a simple linear algebraic argument based only on the commutativity of $L$ and $\overline{L}$.

Since K\"ahler Einstein metrics are trivial generalized K\"ahler Einstein metrics, Theorem \ref{MaLi} includes the following classical Matsushima's result \cite{Mat57} called the Matsushima's obstruction.
\begin{corollary}\label{Matsushima}
Let $X$ be a Fano manifolds admitting a K\"ahler Einstein metric.
Then the holomorphic automorphism group of $X$ is reductive.
\end{corollary}

The work of Yao \cite{Ya} brings many interests to study about related topics of generalized K\"ahler Einstein metrics.
See \cite{N, Na3} (for the modified Ding functional of toric Fano manifolds), \cite{LZ} (for the modified Ding functional of general Fano manifolds), \cite{NSY} (for relative GIT stabilities).
In particular, Collins-Hisamoto-Takahashi \cite{CHT} developed a geometric flow, called the {\it inverse Monge-Amp\`ere flow}, whose self similar solutions are generalized K\"ahler Einstein metrics.
In section \ref{Ding flow}, we discuss a relation between Corollary \ref{Hessian2} and this flow.

\noindent{\bf Acknowledgements.}
The author would like to thank Professor Shigetoshi Bando, Professor Shunsuke Saito and Doctor Ryosuke Takahashi for their several helpful comments and constant encouragement.
\section{Variations of the Ricci Calabi functional}\label{variation}
The Ricci Calabi functional can be written in terms of $\phi\in\mathcal{M}(\omega)$ as
$$\mathcal{E}_{\mathrm{RC}}(\phi)=\int_X(1-e^{f_{\phi}})^2\omega_{\phi}^n.$$
First we compute the variation of the Ricci potential to obtain the first variation of the Ricci Calabi functional.
\begin{lemma}
For any $\phi\in\mathcal{M}(\omega)$ and any direction $\delta\phi\in T_{\phi}\mathcal{M}(\omega)$, we have
$$\delta f_{\phi} = -\Delta_{\phi}\delta\phi-\delta\phi + \frac{1}{V}\int_X \delta\phi e^{f_{\phi}}\omega_{\phi}^n.$$
\end{lemma}
\begin{proof}
By taking the variation of the first equation in \eqref{Riccipotential}, we have
$\delta f_{\phi}=-\Delta_{\phi}\delta\phi -\delta\phi+C$ for some constant $C$.
The constant $C$ is equals to $\frac{1}{V}\int_X \delta\phi e^{f_{\phi}}\omega_{\phi}^n$ 
by the variation $\int_X(\delta f_{\phi} + \Delta\delta\phi)e^{f_{\phi}}\omega_{\phi}^n=0$ of the second equation in \eqref{Riccipotential}.
\end{proof}
\begin{lemma}\label{first variation}{\rm (c.f. \cite[Proof of Theorem 1]{Ya})}
For any $\phi\in\mathcal{M}(\omega)$ and any direction $\delta\phi\in T_{\phi}\mathcal{M}(\omega)$, we have
$$\delta\mathcal{E}_{\mathrm{RC}}(\delta\phi)=2\lla L(e^{f_{\phi}}), \delta\phi\rra = 2\lla \overline{L}(e^{f_{\phi}}),\delta\phi \rra.$$
\end{lemma}
\begin{proof}
By integrations by parts (see also \eqref{Lapf1} and \eqref{Lapf2} in Lemma \ref{L}), note that
\begin{eqnarray*}
\int_X e^{2f_{\phi}}\Delta_{\phi}\delta\phi\omega_{\phi}^n 
&=&\int_X 2(\Delta_{\phi}e^{f_{\phi}} + \langle \dbar e^{f_{\phi}},\dbar f_{\phi}\rangle)\delta\phi e^{f_{\phi}}\omega_{\phi}^n   \\
&=&\int_X 2(\Delta_{\phi}e^{f_{\phi}} + \langle \partial e^{f_{\phi}},\partial f_{\phi}\rangle)\delta\phi e^{f_{\phi}}\omega_{\phi}^n.
\end{eqnarray*}
We then have
\begin{eqnarray*}
\delta\mathcal{E}_{\mathrm{RC}}(\delta\phi)
&=& \int_X e^{2f_{\phi}}2\delta f_{\phi} \omega_{\phi}^n + e^{2f_{\phi}} \Delta_{\phi}\delta\phi\omega_{\phi}^n \\
&=& \int_X 2e^{2f_{\phi}}\Bigl( -\Delta_{\phi}\delta\phi-\delta\phi 
        + \frac{1}{V}\int_X \delta\phi e^{f_{\phi}}\omega_{\phi}^n \Bigr)\omega_{\phi}^n 
        + e^{2f_{\phi}} \Delta_{\phi}\delta\phi\omega_{\phi}^n \\
&=& \int_X 2\Bigl( -\Delta_{\phi}e^{f_{\phi}} -\langle \dbar e^{f_{\phi}},\dbar f_{\phi}\rangle -e^{f_{\phi}} 
        + \frac{1}{V}\int_X e^{f_{\phi}} e^{f_{\phi}}\omega_{\phi}^n \Bigr)\delta\phi e^{f_{\phi}}\omega_{\phi}^n \\
&=& 2\lla L(e^{f_{\phi}}), \delta\phi\rra.
\end{eqnarray*}
Similarly we have
$\delta\mathcal{E}_{\mathrm{RC}}(\delta\phi)= 2\lla \overline{L}(e^{f_{\phi}}),\delta\phi \rra.$
\end{proof}
The followings are fundamental properties for operators $L$ and $\bar{L}$.
\begin{lemma}\label{L}
$L$ and $\overline{L}$ are self-adjoint non-nagative operators on $C^{\infty}_{\mathbb{C}}(X)$ with respect to the inner product $\lla \cdot, \cdot \rra.$
\end{lemma}
\begin{proof}
The operator $u\mapsto \Delta_{\phi}u +\langle\dbar u, \dbar f_{\phi}\rangle$ can be seen as a (nagative) Laplacian with respect to a weighted inner product
$\lla u, v \rra_f := \int_X u\bar{v}e^{f_{\phi}}\omega_{\phi}$ on $C^{\infty}_{\mathbb{C}}(X)$,
and the operator $u\mapsto \Delta_{\phi}u +\langle\partial u,\partial f_{\phi}\rangle$ is its complex conjugate.
Indeed we have
\begin{eqnarray}\label{Lapf1}
\lla \Delta_{\phi}u +\langle\dbar u, \dbar f_{\phi}\rangle, v \rra_f 
&=&\int_X(-\langle \dbar u, \dbar(ve^{f_{\phi}})\rangle + \langle \dbar u, \dbar(e^{f_{\phi}})\rangle)\omega_{\phi}^n \\
&=&\int_X -\langle \dbar u, \dbar v \rangle e^{f_{\phi}}\omega_{\phi}^n \nonumber,
\end{eqnarray}
and
\begin{eqnarray}\label{Lapf2}
\lla u, \Delta_{\phi}v +\langle\dbar v, \dbar f_{\phi}\rangle \rra_f 
&=&\int_X(-\langle \dbar (ue^{f_{\phi}}), \dbar v\rangle + \langle \dbar e^{f_{\phi}}, \dbar v \rangle)\omega_{\phi}^n \\
&=&\int_X -\langle \dbar u, \dbar v \rangle e^{f_{\phi}}\omega_{\phi}^n \nonumber.
\end{eqnarray}
It follows that $L$ and $\overline{L}$ are self-adjoint operators with respect to $\lla \cdot, \cdot \rra.$
Furthermore the following Bochner type formula \cite[Proof of Theorem 2.4.3]{Fubook} (see also \cite[Proposition 2]{Ya}) holds:
$$\int_X | \Delta_{\phi}u +\langle\dbar u, \dbar f_{\phi}\rangle |^2e^{f_{\phi}}\omega_{\phi}^n
= \int_X | \nabla_{\bar{i}}\nabla_{\bar{j}}u |^2e^{f_{\phi}}\omega_{\phi}^n
+\int_X | \dbar u |^2e^{f_{\phi}}\omega_{\phi}^n.$$
Thus the first eigenvalue of the operator $u\mapsto -\Delta_{\phi}u -\langle\dbar u, \dbar f_{\phi}\rangle$ is greater than or equal to $1$.
It follows that $L$ and $\overline{L}$ are non-negative operators with respect to $\lla \cdot, \cdot \rra.$
\end{proof}
The following lemma is crucial for Theorem \ref{Hessian1}.
\begin{lemma}\label{Holomorphicity}
There is an vector space isomorphism:
$$\Set{u\in C^{\infty}_{\mathbb{C}}(X) | Lu=0}/\mathbb{C} \simeq \mathfrak{h}(X).$$
In particular, the isomorphism is given by taking the gradient $u\mapsto\mathrm{grad}_{\phi}u$.
\end{lemma}
\begin{proof}
By definition of $L$, the space $\Set{u\in C^{\infty}_{\mathbb{C}}(X) | Lu=0}/\mathbb{C}$ can be identified with 
$$\Set{u\in C^{\infty}_{\mathbb{C}}(X) | \Delta_{\phi}u +\langle\dbar u, \dbar f_{\phi}\rangle=u}$$
through $u\mapsto u-\frac{1}{V}\int_Xue^{f_{\phi}}\omega_{\phi}^n$.
By \cite[Theorem 2.4.3]{Fubook} (see also \cite[Corollary 4]{Ya}), this can be identified with the space of holomorphic vector fields on $X$ through $u\mapsto\mathrm{grad}_{\phi}u$.
\end{proof}
\begin{remark}
In views of the above lemmas for the operator $L$, 
it is natural to consider the operator $u \mapsto -\Delta_{\phi}u -\langle \dbar u, \dbar f_{\phi}\rangle -u + \frac{1}{V}\int_X u e^{f_{\phi}}\omega_{\phi}^n$ on the vector space with the weighted inner product $(C^{\infty}_{\mathbb{C}}(X), \lla\cdot,\cdot\rra_f)$ instead of the operator $L$ on $(C^{\infty}_{\mathbb{C}}(X), \lla\cdot,\cdot\rra)$.
However, it is technically essential to consider $L$ on $(C^{\infty}_{\mathbb{C}}(X), \lla\cdot,\cdot\rra)$ in the proofs of Theorem \ref{Hessian1} and Theorem \ref{MaLi}. 
\end{remark}
The Lemma \ref{first variation} and Lemma \ref{Holomorphicity} show that  
every critical point of the Ricci Calabi functional defines a generalized K\"ahler Einstein metric and vice versa (c.f. \cite[Theorem 1]{Ya}). 
Indeed, every critical point $\phi\in\mathcal{M}(\omega)$ of $\mathcal{E}_{\mathrm{RC}}$ satisfies $L(e^{f_{\phi}})=0$, and $e^{f_{\phi}}$ defines a holomorphic vector field on $X$.

We now prove Theorem \ref{Hessian1}.

\noindent{\bf Proof of Theorem \ref{Hessian1}.}
We compute the variation $(\delta L)(e^{f_{\phi}})$ at a generalized K\"ahler Einstein metric $\phi\in\mathcal{M}(\omega)$ to obtain the Hessian
$$\mathrm{Hess}(\mathcal{E}_{\mathrm{RC}})(\delta\phi_1, \delta\phi_2)
= 2\lla \delta_1(L(e^{f_{\phi}})), \delta\phi_2 \rra
= 2\lla (\delta_1 L)(e^{f_{\phi}})+L(\delta_1 e^{f_{\phi}}), \delta\phi_2 \rra,$$
where $\delta_1$ means the variation along $\delta\phi_1\in T_{\phi}\mathcal{M}(\omega)$.
Since $\omega_{\phi}=\omega+\deldel\phi$ is generalized K\"ahler Einstein, 
$Z:=\mathrm{grad}_{\phi}e^{f_{\phi}}$ is holomorphic.
By Lemma \ref{Holomorphicity}, 
$e^{f_{\phi}}$ is in the Kernel of $L$ for the K\"ahler metric $\omega_{\phi}$.
Note that $Z$ can be also written by $\mathrm{grad}_{\phi+t\delta\phi}(e^{f_{\phi}}+tZ(\delta\phi))$  for any small $t\in (-\e, \e)$,
since  it is easy to see that $$i_Z(\omega_{\phi}+t\deldel\delta\phi)=\sqrt{-1}\dbar(e^{f_{\phi}}+tZ(\delta\phi)).$$
Thus, by Lemma \ref{Holomorphicity} again, a perturbation $e^{f_{\phi}}+tZ(\delta\phi)$ is in the Kernel of $L$ for the K\"ahler metric $\omega_{\phi}+t\deldel \delta\phi$.
When we denote the operator $L$ for the K\"ahler metric $\omega_{\phi}+t\deldel \delta\phi$ by $L_t$,
we thus have $$L_t(e^{f_{\phi}}+tZ(\delta\phi))=0.$$
Taking derivative at $t=0$, we have
$$(\delta L)(e^{f_{\phi}})=-L(Z(\delta\phi))=-L \langle \partial\delta\phi, \partial(e^{f_{\phi}}) \rangle.$$
Therefore we obtain
\begin{eqnarray*}
(\delta L)(e^{f_{\phi}})+L(\delta e^{f_{\phi}})
&=& L\Bigl( -\langle \partial\delta\phi, \partial f_{\phi} \rangle e^{f_{\phi}} 
+\Bigl( -\Delta_{\phi}\delta\phi -\delta\phi + \frac{1}{V}\int_X \delta\phi e^{f_{\phi}}\omega_{\phi}^n \Bigr)e^{f_{\phi}}  \Bigr) \\
&=& L\overline{L}(\delta\phi).
\end{eqnarray*}
Similarly we have
$\delta(\overline{L}(e^{f_{\phi}}))=\overline{L}L(\delta\phi).$
This completes the proof.
\qed

\noindent{\bf Proof of Corollary \ref{Hessian2}}.
This is an immediate consequence of Theorem \ref{Hessian1} and Lemma \ref{L}.
\qed

\section{Matsushima's type decomposition theorem}
As an application of Corollary \ref{Hessian2}, we give another proof of Theorem \ref{MaLi} which was originally proved by Mabuchi \cite[Theorem 4.1]{Ma01}.
In Mabuchi's proof, it was essential to apply the strict periodicity \cite[Theorem F]{FM} for the real part of the extremal K\"ahler vector field $\mathrm{grad}_{\phi}\mathrm{pr}(S(\omega_{\phi})-\overline{S})$,
where we let
$$\mathrm{pr} : L^2(X, \omega_{\phi}) \to\Set{f\in C^{\infty}_{\mathbb{R}}(X) 
| \dbar \mathrm{grad}_{\phi}f=0 \quad\text{and}\quad \int_Xf\omega_{\phi}^n=0}$$
be the projection, and we let $S(\omega_{\phi})$ be the scalar curvature of $\omega_{\phi}$ and $\overline{S}$ its average.
The vector field $-\mathrm{grad}_{\phi}e^{f_{\phi}}$ in Theorem \ref{MaLi} is nothing but the extremal K\"ahler vector field.
Indeed, we have $1-e^{f_{\phi}}=\mathrm{pr}(1-e^{f_{\phi}})$ by definition of the generalized K\"ahler Einstein metric, 
and we have $\mathrm{pr}(1-e^{f_{\phi}})=\mathrm{pr}(S(\omega_{\phi})-\overline{S})$ by \cite[Theorem 2.1]{Ma01}.

We now prove Theorem \ref{MaLi} by only applying the commutativity of $L$ and $\overline{L}$.

\noindent{\bf Proof of Theorem \ref{MaLi}.}
Since operators $L$ and $\overline{L}$ are commutative by Corollary \ref{Hessian2},
we see that $\overline{L}\in\End(\Ker L)$.
Let $E_{\lambda}$ be the $\lambda$-eigenspace of $\overline{L}|_{\Ker L}$.
Note that $\lambda\geq 0$, since the operator $\overline{L}$ is non-negarive.
Let us take any $u\in E_{\lambda}$.
Since $E_{\lambda}\subset \Ker L$, we have $\mathrm{grad}_{\phi}u \in \mathfrak{h}(X).$
We also have
\begin{eqnarray*}
\lambda u &=& \overline{L}u \\
&=& (\overline{L} -L)u \\
&=& -\langle \partial u, \partial e^{f_{\phi}} \rangle + \langle \dbar u, \dbar e^{f_{\phi}} \rangle \\
&=& \{ e^{f_{\phi}}, u \},
\end{eqnarray*}
where $\{ \cdot, \cdot \}$ is the Poisson bracket defined by $\{ u, v \} = (\mathrm{grad}_{\phi}v)u-(\mathrm{grad}_{\phi}u)v$.
It is well-known that the map $u\mapsto -\mathrm{grad}_{\phi}u$ is a complex Lie algebra homomorphism 
from $(C^{\infty}_{\mathbb{C}}, \{ \cdot, \cdot \})$ to $(\Gamma(TX), [ \cdot,\cdot ])$, 
where $[\cdot, \cdot ]$ is the Lie bracket defined by $[Z,W]=ZW-WZ$.
Thus it follows that
$$\lambda\mathrm{grad}_{\phi}u=[-\mathrm{grad}_{\phi}e^{f_{\phi}}, \mathrm{grad}_{\phi}u].$$
Let $\mathfrak{h}_{\lambda}(X):=\mathrm{grad}_{\phi}(E_{\lambda})$.
Since by Lemma \ref{Holomorphicity}, every holomorphic vector field on $X$ can be written as $\mathrm{grad}_{\phi}u$ for a function $u\in\Ker L$ unique up to additive constant,
we then have a decomposition
$$\mathfrak{h}(X)=\sum_{\lambda\geq 0}\mathfrak{h}_{\lambda}(X).$$
By the above computation, we see that each $\mathfrak{h}_{\lambda}(X)$ is 
the $\lambda$-eigenspae of the action of $-\mathrm{ad}(\mathrm{grad}_{\phi}e^{f_{\phi}})$.

Since $E_0=\Ker L \cap \Ker\overline{L}$, if $u\in E_0$, then both the real part and the imaginary part of $u$ are in $E_0$.
It follows that $$E_0=\Set{u\in\sqrt{-1}C^{\infty}_{\mathbb{R}}(X) | \mathrm{grad}_{\phi}u \in \mathfrak{h}(X) }\otimes\mathbb{C}.$$
Thus $$\mathfrak{h}_0(X)=\Set{\mathrm{grad}_{\phi}u+\overline{\mathrm{grad}_{\phi}u} | u\in\sqrt{-1}C^{\infty}_{\mathbb{R}}(X) \text{ and } \mathrm{grad}_{\phi}u \in \mathfrak{h}(X)}\otimes\mathbb{C}.$$
By \cite[Lemma 2.3.8]{Fubook}, $\Set{\mathrm{grad}_{\phi}u+\overline{\mathrm{grad}_{\phi}u} | u\in\sqrt{-1}C^{\infty}_{\mathbb{R}}(X) \text{ and } \mathrm{grad}_{\phi}u \in \mathfrak{h}(X)}$ is equal to the space of Killing vector fields on $X$.
Since the Lie algebra of Killing vector fields on $X$ corresponds to the Isometry group of $X$ which is a compact group, we see that $\mathfrak{h}_0(X)$ is reductive.
This completes the proof.
\qed

\noindent{\bf Proof of Corollary \ref{Matsushima}.}
For any K\"ahler Einstein metric $\phi\in\mathcal{M}(\omega)$, the holomorphic vector field $\mathrm{grad}_{\phi}e^{f_{\phi}}$ vanishes.
By Theorem \ref{MaLi}, we thus have $\mathfrak{h}(X)=\mathfrak{h}_0(X)$.
\qed

\section{The inverse Monge-Amp\`ere flow and generalized K\"ahler Einstein metrics}\label{Ding flow}
Very recently, Collins-Hisamoto-Takahashi \cite{CHT} developed a geometric flow $\{\phi_t\}_{t\in[0,\infty)}\subset\mathcal{M}(\omega)$, 
called the {\it inverse Monge-Amp\`ere flow} (written as the MA$^{-1}$ flow for simplicity), defined by
$$\frac{d\phi_t}{dt}=1-e^{f_{\phi_t}}.$$ 
They proved the existence of a long time solution for any initial K\"ahler metric.
It is easy to see that self-similar solutions of the MA$^{-1}$ flow are generalized K\"ahler Einstein metrics.
Furthermore the Ricci Calabi functional $\mathcal{E}_{\mathrm{RC}}$ is monotonically decreasing along the MA$^{-1}$ flow.
Indeed, by Lemma \ref{first variation} and Lemma \ref{L} we have
$$\frac{d}{dt}\mathcal{E}_{\mathrm{RC}}(\phi_t)=-2\lla L(e^{f_{\phi_t}}), e^{f_{\phi_t}} \rra \leq 0,$$ along the flow $\phi_t$.
Therefore, Corollary \ref{Hessian2} suggests that for any Fano manifold admitting a generalized K\"ahler Einstein metric, the MA$^{-1}$ flow with any initial K\"ahler metric converges to a generalized K\"ahler Einstein metric in some sense.


\bigskip

\address{
Mathematical Institute \\ 
Tohoku University \\
Sendai 980-8578 \\
Japan
}
{satoshi.nakamura.r8@dc.tohoku.ac.jp}
\end{document}